\documentclass[12pt]{amsart}
\usepackage{amsmath}
\usepackage{amssymb}
\usepackage{amsfonts}
\usepackage{amsthm}
\usepackage{array}

\theoremstyle{plain}
\numberwithin{equation}{section}
\newtheorem{Theorem}{Theorem}
\newtheorem{Lemma}[Theorem]{Lemma}
\newtheorem{Proposition}[Theorem]{Proposition}

\theoremstyle{remark}

\date{}

\title[Improved asymptotics]{Improved asymptotics of
the spectral gap for the Mathieu operator}

\author{Berkay Anahtarci}

\author{Plamen Djakov}

\address{Sabanci University, Orhanli,
34956 Tuzla, Istanbul, Turkey}
 \email{berkaya@sabanciuniv.edu}

\address{Sabanci University, Orhanli,
34956 Tuzla, Istanbul, Turkey}
 \email{djakov@sabanciuniv.edu}

\begin{document}

\begin{abstract}
The  Mathieu operator
\begin{equation*}
L(y)=-y''+2a \cos{(2x)}y, \quad a\in \mathbb{C},\;a\neq 0,
\end{equation*}
considered with periodic or anti-periodic boundary conditions has,
close to $n^2$ for large enough $n$, two periodic (if $n$ is even) or
anti-periodic (if $n$ is odd) eigenvalues $\lambda_n^-$,
$\lambda_n^+$. For fixed $a$, we show that
\begin{equation*}
\lambda_n^+ - \lambda_n^-= \pm \frac{8(a/4)^n}{[(n-1)!]^2} \left[1 -
\frac{a^2}{4n^3}+ O \left( \frac{1}{n^4}\right)\right], \quad
n\rightarrow\infty.
\end{equation*}
This result extends the asymptotic formula of Harrell-Avron-Simon, by
providing more asymptotic terms.
\end{abstract}

\maketitle

\section{Introduction}

The one-dimensional Schr\"{o}dinger operator
\begin{equation}
\label{i1} L(y)=-y''+v(x)y,
\end{equation}
considered on $\mathbb{R}$ with $\pi$-periodic real-valued potential
$v\in L^2_{loc}(\mathbb{R})$, is self-adjoint, and its spectrum has a
gap-band structure (see Thm 2.3.1 in  \cite{E}, or Thm 2.1 in
\cite{MW}); namely, there are points
\begin{equation*}
\lambda_0^+ < \lambda_1^- \leq \lambda_1^+ < \lambda_2^- \leq
\lambda_2^+ < \lambda_3^- \leq \lambda_3^+ < \lambda_4^- \leq \lambda_4^+ <\cdots
\end{equation*}
such that
\begin{equation*}
Sp(L)= \bigcup_{n=1}^\infty [\lambda_{n-1}^{+},\lambda_n^{-}]
\end{equation*}
and the intervals of the spectrum are separated by the spectral gaps
\begin{equation*}
(-\infty,\lambda_0^+), (\lambda_1^-,\lambda_1^+),
\ldots,(\lambda_n^-,\lambda_n^+),\ldots.
\end{equation*}

The points $\lambda^-_n, \lambda^+_n $ could be determined as
eigenvalues of the Hill equation
\begin{equation}
\label{i02}
 - y^{\prime \prime} + v(x) y
= \lambda y,
\end{equation}
considered on $[0,\pi],$ respectively,  with periodic boundary
conditions
\begin{equation}
\label{i03}   y(0) = y(\pi), \quad y^\prime (0) = y^\prime (\pi),
\end{equation}
for even $n,$ and with anti-periodic boundary conditions
\begin{equation}
\label{i04}    y(0) =- y(\pi), \quad y^\prime (0) = - y^\prime (\pi),
\end{equation}
for odd $n.$  See basics and details in \cite{E,MW,Mar}.\\

The rate of decay of the size of the spectral gap $\gamma_n =
\lambda_n^+ - \lambda_n^- $ is closely related to the smoothness of
the potential $v.$ Here we mention only Hochstadt's result
\cite{Hoch63} that an $L^2 ([0,\pi])$-potential $v$ is in $C^\infty $
if and only if $(\gamma_n )$ decays faster than any power of $1/n, $
and Trubowitz's result \cite{Tr} that an $L^2 ([0,\pi])$-potential
$v$ is analytic if and only if $(\gamma_n )$ decays exponentially.

If $v$ is a complex-valued potential then the operator (\ref{i1}) is
non-self-adjoint, so one cannot talk about spectral gaps. Moreover,
the periodic and anti-periodic eigenvalues $\lambda^\pm_n $ are
well-defined for large $n$ (see Lemma~\ref{loc} below)  but the
asymptotics of $|\lambda_n^+ - \lambda_n^-|$ does not determine the
smoothness of $v.$ In \cite{Tk92} V. Tkachenko suggested to consider
also the Dirichlet b.v.p. $y(\pi)=y(0)=0. $  For large enough $n, $
close to $n^2 $ there is exactly one Dirichlet eigenvalue $\mu_n, $
so the {\em deviation}
\begin{equation}
\label{0.2}  \delta_n = |\mu_n - \frac{1}{2}(\lambda^+_n +
\lambda^-_n)|
\end{equation}
is well defined. Using an adequate parametrization of potentials in
spectral terms similar to Marchenko--Ostrovskii's ones
\cite{Mar,MO75} for self-adjoint operators, V. Tkachenko
\cite{Tk92,Tk94}  (see also \cite{ST}) characterized
$C^\infty$-smoothness and analyticity in terms of $\delta_n$ and
differences between critical values of Lyapunov functions and
$(-1)^n$. See further references and later results in \cite{DM3, DM5, DM15}.\\

In the case of specific potentials, like the Mathieu potential
\begin{equation}
\label{i05} v(x) = 2a \cos 2x, \quad a\neq 0, \;\text{real},
\end{equation}
or more general trigonometric polynomials
\begin{equation}
\label{i06} v(x) = \sum_{-N}^N c_k \exp (2ikx), \quad c_k =
\overline{c_{-k}}, \quad 0 \leq k \leq N < \infty,
\end{equation}
one comes to two classes of questions:

(i) Is the $n$-th spectral gap closed, i.e.,
\begin{equation}
\label{i07} \gamma_n = \lambda_n^+ - \lambda_n^- = 0,
\end{equation}
or, equivalently, is the multiplicity of $\lambda_n^+ $ equal to 2?

(ii) If $ \gamma_n \neq 0, $ could we tell more about the size of
this gap, or, for large enough $n,$ what is the asymptotic behavior
of $\gamma_n = \gamma_n (v)?$

E. L. Ince \cite{I} proved that the Mathieu-Hill operator has only
{\em simple } eigenvalues both for periodic and anti-periodic
boundary conditions, i.e., $\gamma_n \neq 0 $ for every $n \in
\mathbb{N}. $ His proof is presented in \cite{E}; see other proofs of
this fact in \cite{Hil,Mark,McL}, and further references in
\cite{E,MW58}.\\

For fixed $n$ and $a \to 0, $ D. Levy and J. Keller \cite{LK} gave an
asymptotics of the spectral gap $ \gamma_n = \gamma_n (a),\; v \in
(\ref{i05});$ namely
\begin{equation}
\label{i08} \gamma_n = \lambda_n^+ - \lambda_n^- =
\frac{8(|a|/4)^n}{[(n-1)!]^2} \left ( 1+ O(a) \right ), \quad a \to
0.
\end{equation}
Almost 20 years later, E. Harrell \cite{Har} found, up to a constant
factor, the asymptotics of the spectral gaps of the Mathieu operator
for fixed $a$ as  $ n\to \infty. $  J. Avron and B. Simon \cite{AS}
gave an alternative proof of E. Harrell's asymptotics and found the
exact value of the constant factor, which led to the formula
\begin{equation}
\label{i09} \gamma_n = \lambda_n^+ - \lambda_n^- =
\frac{8(|a|/4)^n}{[(n-1)!]^2} \left ( 1+ O\left (\frac{1}{n^2}\right
) \right ) \quad n \to \infty.
\end{equation}
Later, another proof of (\ref{i09}) was given by H. Hochstadt
\cite{Hoch84}. For general trigonometric polynomial potentials,
A. Grigis \cite{Gr} obtained a generic form of the main
term in the gap asymptotics.\\

In this paper we extend the result of Harrell-Avron-Simon and give
the following more precise asymptotics of the size of spectral gap
for the Mathieu operator (even in the case when the parameter $a$ is
a complex number):
\begin{equation}
\label{BF} \lambda_n^+ - \lambda_n^-= \pm \frac{8(a/4)^n}{[(n-1)!]^2}
\left[1 - \frac{a^2}{4n^3}+ O \left( \frac{1}{n^4}\right)\right],
\quad n\rightarrow\infty.
\end{equation}

Our approach is based on the methods developed in \cite{DM11,DM10},
where the gap asymptotics of the Hill operator with two term
potential of the form
$$ v(x) = A \cos 2x + B \cos 4x, \quad A\neq 0, \; B \neq 0
$$
was found. The same methods and similar asymptotic estimates play a
crucial role in the study of Riesz basis property of the root system
of Hill operators with trigonometric polynomial potentials (see
\cite{DM25}).\\

Let us mention  that the paper \cite{Kirac} claims to provide "the
formula which states the isolated terms of arbitrary number in the
asymptotics of the sequence $\gamma_n.$"  However, this claim is
false due to a unavoidable technical mistake (in \cite{Kirac},
formula (5) does not imply (4) for $m=k+1$ since, by Stirling's
formula, the remainder in (5) is much larger than the main term in
(4)).

To the best of our knowledge, (\ref{BF}) is the first formula that
gives more asymptotic terms than the formula of Harrell-Avron-Simon.

\section{Preliminaries}

Let $L_{Per^+}(v)$ and $L_{Per^-}(v)$ denote, respectively, the
operator (\ref{i1}) considered with periodic $(Per^+)$ or
antiperiodic $(Per^-)$ boundary conditions. Further we assume that
$v\in L^2 ([0,\pi])$ is a complex-valued potential such that
\begin{equation}
\label{p1} V(0) = \int_0^\pi v(x) dx =0.
\end{equation}
The following assertion is well-known (e.g., \cite[Proposition
1]{DM10}).
\begin{Lemma}
\label{loc} The spectra of $L_{Per^\pm} (v)$ are discrete. There is
an $N_0=N_0 (v)$ such that the union $\cup_{n>N} D_n $ of the discs
$D_n =\{z: \, |z-n^2|< 1 \}$ contains all but finitely many of the
eigenvalues of $L_{Per^\pm}.$

Moreover, for $n>N$ the disc $D_n$ contains exactly two (counted with
algebraic multiplicity) periodic (if $n$ is even) or antiperiodic (if
$n$ is odd) eigenvalues $\lambda_n^-, \lambda_n^+$  (where $Re
\,\lambda_n^- < Re \,\lambda_n^+$  or $Re \,\lambda_n^- = Re
\,\lambda_n^+$ and $Im \,\lambda_n^- \leq Im \,\lambda_n^+).$
\end{Lemma}

In view of Lemma \ref{loc},
\begin{equation}
|\lambda_n^{\pm}-n^2|<1, \quad \text{for } n\geq N_0.
\end{equation}
 Moreover, Lemma \ref{loc} allows
us to apply the Lyapunov--Schmidt projection method and reduce the
eigenvalue equation $Ly = \lambda y $ for  $\lambda \in D_n $ to an
eigenvalue equation in the two-dimensional space $E_n^0 = \{ L^0 y =
n^2 y\}$   (see \cite[Section 2.2]{DM15}).

This leads to the following (see the formulas (2.24)--(2.30) in
\cite{DM15}).
\begin{Lemma}
\label{lem1} In the above notations, $\lambda_n^\pm=n^2 + z$, for
$|z|<1$, is an eigenvalue of $L_{Per^\pm} (v)$ if and only if $z$ is
a root of the equation
\begin{equation}
\label{determinant S-z}
\begin{vmatrix} z-S^{11} & S^{12} \\ S^{21} & z-S^{22} \end{vmatrix} =0,
\end{equation}
where $S^{11},S^{12},S^{21},S^{22}$ can be represented as
\begin{equation}\label{S_ij}
S^{ij}(n,z)= \sum_{k=0}^{\infty}S_k^{ij}(n,z), \quad i,j=1,2,
\end{equation}
with
\begin{equation}\label{S_0}
S_0^{11}= S_0^{22}=0, \quad S_0^{12}=V(-2n), \quad S_0^{21}=V(2n),
\end{equation}
and for each $k=1,2,...,$
\begin{align}
\label{S_11}  S_k^{11}(n,z) &= \sum_{j_1,...,j_k\neq\pm n}
\frac{V(-n+j_1)V(j_2-j_1)\cdots V(j_{k}-j_{k-1})V(n-j_k)}{(n^2-j_1^2+z)\cdots(n^2-j_k^2+z)},\\
\label{S_22}  S_k^{22}(n,z) &= \sum_{j_1,...,j_k\neq\pm n}
\frac{V(n+j_1)V(j_2-j_1)\cdots V(j_k -j_{k-1})V(-n-j_k)}{(n^2-j_1^2+z)\cdots(n^2-j_k^2+z)},\\
\label{S_12}   S_k^{12}(n,z) &= \sum_{j_1,...,j_k\neq\pm n}
\frac{V(-n+j_1)V(j_2-j_1)\cdots V(j_k -j_{k-1})V(-n- j_k)}{(n^2-j_1^2+z)\cdots(n^2-j_k^2+z)},\\
\label{S_21} S_k^{21}(n,z) &= \sum_{j_1,...,j_k\neq\pm n}
\frac{V(n+j_1)V(j_2-j_1)\cdots
V(j_k-j_{k-1})V(n-j_k)}{(n^2-j_1^2+z)\cdots(n^2-j_k^2+z)}.
\end{align}
The above series converge absolutely and uniformly for $|z|\leq 1$.
\end{Lemma}

Moreover, (\ref{S_ij})--(\ref{S_21}) imply the following (see
Lemma~23 in \cite{DM15}).
\begin{Lemma} For any (complex-valued) potential $v$
\begin{equation}
\label{i} S^{11}(n,z)=S^{22}(n,z).
\end{equation}
 Moreover, if $V(-m)=\overline{V(m)}\;\; \forall m,$ then
\begin{equation}\label{ii}
 S^{12}(n,z)=\overline{S^{21}(n,\bar{z})},
\end{equation}
and if $V(-m)=V(m) \;\; \forall m,$ then
\begin{equation} \label{iii}
S^{12}(n,z)=S^{21}(n,z).
\end{equation}
\end{Lemma}

\begin{proof}
For each $k\in \mathbb{N}, $ the change of summation indices $i_s =-
j_{k+1-s}, $ $ s=1, \ldots, k$ proves that $S^{11}_k (n,z) = S^{22}_k
(n,z). $  In view of (\ref{S_ij}) and (\ref{S_0}), (\ref{i}) follows.

In a similar way,  we obtain that (\ref{ii}) and (\ref{iii}) hold by
using for each $k\in \mathbb{N} $ the change of indices $i_s=
j_{k+1-s},$ $ s=1,2,\ldots,k.$
\end{proof}

Further we consider only the Mathieu potential, i.e.,
\begin{equation}
\label{mp} v(x)= 2a \cos 2x = a e^{-2ix} + a e^{2ix}, \quad V(\pm
2)=a, \;\; V(k)=0 \;\;\text{if} \; k \neq \pm 2.
\end{equation}
For convenience, we set
\begin{equation}
\label{ab}
\alpha_n(z):=S^{11}(n,z)=S^{22}(n,z),\quad \beta_n(z):=S^{21}(n,z)=S^{12}(n,z).
\end{equation}
In these notations the basic equation  \eqref{determinant S-z}
becomes
\begin{equation}
\label{be}
(z-\alpha_n(z))^2=(\beta_n(z))^2.
\end{equation}
By Lemmas \ref{loc} and \ref{lem1},  for large enough $n\in
\mathbb{N},$ this equation has in the unit disc exactly the following
two roots (counted with multiplicity):
\begin{equation}
\label{zn}  z_n^- = \lambda_n^- -n^2, \quad z_n^+ = \lambda_n^+ -n^2.
\end{equation}

\section{Asymptotic Estimates for $z_n^\pm$ and $\alpha_n (z).$}

In this section we use the basic equation (\ref{be})  to derive
asymptotic estimates for the deviations $z_n^\pm.$ It turns out that
$|\beta_n (z)|, \; |z|\leq 1, $ is much smaller than $|\alpha_n (z)|,
$  so it is enough to analyze the asymptotics of $\alpha_n (z_n^\pm)
$ in order to find asymptotic estimates for $z_n^\pm.$

The following inequality is well known (e.g., see Lemma~78 in
\cite{DM15}):
\begin{equation}
\label{In1} \sum_{j \neq \pm n} \frac{1}{|n^2-j^2|} < \frac{2\log
6n}{n}, \quad \text{for} \; n \in \mathbb{N}.
\end{equation}

\begin{Lemma}
\label{lem2}
 If $|z|\leq 1$, then
\begin{equation}\label{In2}
 \sum_{j_1,...,j_\nu\neq\pm n} \frac{1}{|n^2-j_1^2+z|\cdots |n^2-j_\nu^2+z|} <
  \left( \frac{4\log 6n}{n} \right)^\nu.
\end{equation}
\end{Lemma}

\begin{proof} If $|z|\leq 1$ and $j\neq \pm n$, then
\begin{equation*}
|n^2-j^2+z| \geq |n^2-j^2|-1 \geq \frac{1}{2} |n^2-j^2|.
\end{equation*}
Therefore,
\begin{align*}
\sum_{j_1,...,j_\nu\neq\pm n} \frac{1}{|n^2-j_1^2+z|\cdots
|n^2-j_\nu^2+z|} \leq 2^\nu \left(\sum_{j \neq \pm n}
\frac{1}{|n^2-j^2|} \right)^\nu,
\end{align*}
so (\ref{In2}) follows from (\ref{In1}).
\end{proof}

The next lemma  gives a rough estimate for $\beta_n(z);$ we improve
this estimate in the next section.
\begin{Lemma}
\label{lem3} For $|z|\leq 1  $ we have
\begin{equation}
\label{BE1} \beta_n(z)=O\left( \left( \frac{\log n}{n} \right)^n
\right).
\end{equation}
\end{Lemma}

\begin{proof}
If $\nu < n-1, $ then all terms of the sum $S^{21}_\nu (n,z) $ in
(\ref{S_21}) vanish. Indeed, each term of the sum $S^{21}_\nu (n,z) $
is a fraction which nominator has the form $V(x_1) V(x_2) \cdots
V(x_{\nu+1}) $ with $ x_1= n+j_1, \; x_2 = j_2-j_1,  \ldots,  x_{\nu
+1}= n- j_{\nu}. $ Therefore, if $\nu < n-1$ then there are no $x_1,
x_2,\ldots, x_{\nu+1} \in \{-2,2\}$ satisfying
$x_1+x_2+\cdots+x_{\nu+1}=2n,$ so every term of the sum $S^{21}_\nu
(n,z) $ vanishes due to (\ref{mp}). Hence, by (\ref{mp}) we have
\begin{equation*}
|\beta_n(z)| \leq \sum_{\nu=n-1}^\infty \sum_{j_1,...,j_\nu \neq\pm
n} \frac{|a|^{\nu+1}}{|n^2-j_1^2+z|\cdots |n^2-j_\nu^2+z|},
\end{equation*}
so (\ref{BE1}) follows from (\ref{In2}).
\end{proof}

\begin{Lemma}
\label{lemA} In the above notations,
\begin{equation}
\label{In0} z_n^\pm = \frac{a^2}{2n^2} + O\left( \frac{1}{n^4}
\right), \quad  \alpha_n (z_n^\pm)=\frac{a^2}{2n^2} + O\left(
\frac{1}{n^4} \right), \quad  n \to \infty.
\end{equation}
\end{Lemma}

\begin{proof}
In view of  (\ref{S_ij}), (\ref{S_11}) and (\ref{ab}), we have
\begin{equation}
\label{A} \alpha_n(z)= \sum_{p=1}^\infty A_p(n,z),
\end{equation}
where
\begin{equation}\label{A_p}
A_p(n,z)= \sum_{j_1,...,j_p\neq\pm n} \frac{V(-n+j_1)V(j_2-j_1)\cdots
V(j_p - j_{p-1})V(n-j_p)}{(n^2-j_1^2+z)\cdots(n^2-j_p^2+z)}.
\end{equation}

First we show that
\begin{equation}
\label{A2k} A_{2k}(n,z) \equiv 0 \quad \forall k \in \mathbb{N}.
\end{equation}
Indeed, for $p=2k$ each term of the sum in (\ref{A_p}) is a fraction
which nominator has the form $V(x_1) V(x_2) \cdots V(x_{2k+1}) $ with
\begin{equation*}
x_1= -n+j_1, \quad x_2= j_2-j_1, \quad \ldots, \quad x_{2k+1}= n-
j_{2k}.
\end{equation*}
Since  $x_1+x_2+\cdots+x_{2k+1}=0, $ it follows that there is $i_0$
such that $x_{i_0} \neq \pm 2, $ so $V(x_{i_0}) =0 $ due to
(\ref{mp}). Therefore, every term of the sum $A_{2k} (n,z) $
vanishes, hence (\ref{A2k}) holds.\\

Next we estimate iteratively, in two steps,   $\alpha_n (z)$ and
$z_n^\pm.$ The first step provides rough estimates which we improve
in the second step.\\

{\em Step 1.} By (\ref{A_p}), we have
\begin{equation*}
A_1(n,z)= \sum_{j_1\neq \pm n}\frac{V(-n+j_1)V(n-j_1)}{n^2-j_1^2+z}.
\end{equation*}
 In view of (\ref{mp}), we get a non-zero term in the above sum if and only if
$j_1=n+2,$ or $ j_1=n-2.$ Therefore,
\begin{align}
\label{A_1} A_1 (n,z) = \frac{a^2}{n^2-(n-2)^2+z} +
\frac{a^2}{n^2-(n+2)^2+z} = a^2 \frac{8-2z}{(4n)^2-(4-z)^2},
\end{align}
which implies that
\begin{equation}
\label{In4} A_1(n,z)= O\left(\frac{1}{n^2}\right)  \quad \text{for}
\;\; |z|\leq 1.
\end{equation}

On the other hand, from (\ref{mp}), \eqref{In2} and (\ref{A_p}) it
follows that
\begin{equation}
\label{In5} |A_{2k-1}(n,z)| \leq |a|^{2k}\left( \frac{ 4\log 6n }{n}
\right)^{2k-1}, \quad k=2,3, \ldots,
\end{equation}
which implies
\begin{equation}
\label{In6} \sum_{k=2}^\infty |A_{2k-1}(n,z)| \leq \sum_{k=2}^\infty
|a|^{2k} \left( \frac{4\log 6n}{n} \right)^{2k-1}= o \left(1/n^2
\right).
\end{equation}
Hence, by (\ref{In4}) and (\ref{In6}) we obtain
\begin{equation}
\label{In7} \alpha_n(z)=O\left( \frac{1}{n^2} \right) \quad
\text{for} \;\; |z|\leq 1.
\end{equation}

Furthermore, from (\ref{be}), (\ref{zn}) and (\ref{BE1}) it follows
immediately that
\begin{equation}
\label{In8} z_n^\pm -  \alpha_n(z_n^\pm)=O\left( \frac{1}{n^k}
\right), \quad  \forall k \in \mathbb{N}.
\end{equation}
Therefore, (\ref{In7}) implies that
\begin{equation}
\label{In11} z_n^\pm= O\left(\frac{1}{n^2} \right).
\end{equation}

\emph{Step 2.}    By \eqref{A_1} we have
\begin{equation}\label{In10}
A_1 (n,z)= \frac{a^2}{2n^2}+ O\left(\frac{1}{n^4}\right) \quad
\text{if}  \;\; z=O(1/n^2).
\end{equation}
Let us consider
\begin{equation*}
A_3(n,z)= \sum_{j_1,j_2,j_3\neq \pm n} \frac{V(-n+j_1)V(j_2-j_1)
V(j_3 -j_2) V(n-j_3)}{(n^2-j_1^2+z)(n^2-j_2^2+z)(n^2-j_3^2+z)}.
\end{equation*}
In view of (\ref{mp}), we get a non-zero term in the above sum if and
only if
\begin{equation*}
j_1=n+2; \quad j_2=n+4; \quad j_3=n+2,
\end{equation*}
or
\begin{equation*}
j_1=n-2; \quad j_2=n-4; \quad j_3=n-2.
\end{equation*}
Hence,
\begin{align*}
  A_3(n,z) &=
\frac{a^4}{[n^2-(n+2)^2+z][n^2-(n+4)^2+z][n^2-(n+2)^2+z]}
\\  &+ \frac{a^4}{[n^2-(n-2)^2+z][n^2-(n-4)^2+z][n^2-(n-2)^2+z]},
\end{align*}
so it is easy to see that
\begin{align}\label{In12}
A_3(n,z)=  O\left( \frac{1}{n^4}\right) \quad \text{if} \;\;|z|\leq
1.
\end{align}
On the other hand, by \eqref{In5} we have
\begin{equation}
\label{In13} \sum_{k=3}^\infty |A_{2k-1}(n,z)| \leq \sum_{k=3}^\infty
|a|^{2k} \left( \frac{4\log 6n}{n} \right)^{2k-1}= o \left(1/n^4
\right).
\end{equation}
Therefore, by (\ref{In10}), (\ref{In12}) and (\ref{In13}) imply that
\begin{equation}
\label{In14} \alpha_n(z)=\frac{a^2}{2n^2}+
O\left(\frac{1}{n^4}\right) \quad \text{if} \;\; z=O(1/n^2).
\end{equation}

Hence, from \eqref{In8} it follows that
\begin{equation}
\label{In15}
z_n^\pm = \frac{a^2}{2n^2} + O \left( \frac{1}{n^4} \right).
\end{equation}
\end{proof}

\emph{Remark.}  From \eqref{A_1} and \eqref{In15} it follows that
\begin{equation}
\label{In16} A_1(n,z_n^\pm)= \frac{a^2}{2n^2} + \frac{a^2}{2n^4} -
\frac{a^4}{16n^4} + O\left( \frac{1}{n^6} \right).
\end{equation}
Similarly, it is easily seen that
\begin{equation}
\label{In17} A_3 (n,z_n^\pm)= \frac{a^4}{16n^4} + O\left(
\frac{1}{n^6} \right).
\end{equation}
On the other hand, analyzing $A_5 (n,z) $ one can show that
\begin{equation}
\label{In18} A_5 (n,z)=  O\left( \frac{1}{n^6} \right) \quad
\text{if} \;\; |z|\leq 1.
\end{equation}
Moreover, by (\ref{In5}) we have
\begin{equation}
\label{In19} \sum_{k=4}^\infty |A_{2k-1}(n,z)| = o \left(
\frac{1}{n^6} \right) \quad \text{if} \;\; |z|\leq 1.
\end{equation}
Hence, in view of \eqref{In8}, the estimates
\eqref{In16}--\eqref{In19} lead to
\begin{equation}
\label{BestEstimateZ_n}
z_n^\pm= \frac{a^2}{2n^2} + \frac{a^2}{2n^4} + O\left( \frac{1}{n^6} \right).
\end{equation}
This analysis could be extended in order to obtain more asymptotic
terms of $z_n^\pm, $ and even to explain that the corresponding
asymptotic series along the powers of $1/n$ contains only even
nontrivial terms. However,
in this paper we need only the estimate (\ref{In15}).\\

The following assertion plays an essential role later.
\begin{Lemma} With $\gamma_n = \lambda_n^+ - \lambda_n^- = z_n^+ -
z_n^-,$
\begin{equation}
\label{D1} \alpha_n(z_n^+)-\alpha_n(z_n^-)= \gamma_n \left[
-\frac{a^2}{8n^2} + O \left( \frac{1}{n^4}\right) \right].
\end{equation}
\end{Lemma}

\begin{proof} By \eqref{A} and \eqref{A2k} we obtain
\begin{align}\label{D2}
\alpha_n(z_n^+)-\alpha_n(z_n^-)= A_1(n,z_n^+)- A_1(n,z_n^-) +
\int_{z_n^-}^{z_n^+}  \frac{d}{dz} \tilde{\alpha}_n(z) \, dz,
 \end{align}
where
\begin{align*}
\tilde{\alpha}_n(z)= \alpha_n(z) -A_1(n,z) =A_3(n,z) + A_5(n,z) +
\cdots.
\end{align*}
In view of \eqref{In12} and  \eqref{In13},
\begin{equation*}
\tilde{\alpha}_n(z) = O(1/n^4) \quad \text{for} \;\; |z|\leq 1.
\end{equation*}
By the Cauchy formula for derivatives, this estimate implies that
$$
d\alpha_n/dz =O(1/n^4) \quad \text{for} \;\; |z|\leq 1/2.
$$
Hence, we obtain
\begin{align} \label{D3}
\int_{z_n^-}^{z_n^+}  \frac{d}{dz} \tilde{\alpha}_n(z) \, dz =
\gamma_n O \left( \frac{1}{n^4} \right).
\end{align}
On the other hand, by \eqref{A_1}
\begin{align} \nonumber
A_1(n,z_n^+)-A_1(n,z_n^-) &= \left[ \frac{8-2z_n^+}{(4n)^2-
(4-z_n^+)^2} - \frac{8-2z_n^-}{(4n)^2- (4-z_n^-)^2}  \right] a^2 \\
\nonumber &=\gamma_n \left[\frac{-32n^2  -32 + 8(z_n^+ + z_n^-) -2
z_n^+ z_n^-} {[(4n)^2- (4-z_n^+)^2][(4n)^2- (4-z_n^-)^2]}  \right]
a^2.
\end{align}
Therefore, taking into account \eqref{In0}, we obtain
\begin{equation}
\label{D4} A_1(n,z_n^+)-A_1(n,z_n^-)= \gamma_n \left[
\frac{-a^2}{8n^2} + O \left( \frac{1}{n^4} \right) \right].
\end{equation}
In view of  \eqref{D2}, the estimates \eqref{D3} and \eqref{D4} lead
to \eqref{D1}.
\end{proof}

\section{Asymptotic formulas for $\beta_n(z_n^\pm)$ and $\gamma_n.$}

In this section we find more precise asymptotics of $\beta_n
(z_n^\pm)$. These asymptotics, combined with the results of the
previous section,
lead to an asymptotics for $\gamma_n.$  \\

In view of (\ref{mp}), each nonzero term in \eqref{S_21} corresponds
to a \emph{k}-tuple of indices $(j_1,...,j_k)$ with $j_1, \ldots ,
j_k \neq \pm n$ such that
\begin{equation}
\label{a1} (n+j_1)+(j_2-j_1)+ \cdots +(j_k-j_{k-1})+(n-j_k)=2n
\end{equation}
and
\begin{equation}
\label{a2} n+j_1,  j_2-j_1, \ldots, j_k-j_{k-1},   n-j_k \in \{-2,
2\}.
\end{equation}
By (\ref{a1}) and (\ref{a2}), there is one-to-one correspondence
between the nonzero terms in \eqref{S_21} and the "admissible" walks
$x=(x(t))_{t=1}^{k+1}$ on $\mathbb{Z}$ from $-n$ to $n$ with steps
$x(t)= \pm 2$ and vertices $j_0= -n, \; j_{k+1}=n, $
\begin{equation}
\label{j_k}  j_s=-n+\sum_{t=1}^s x(t) \neq \pm n, \quad s=1,\ldots,
k.
\end{equation}

Let $X_n(p), \; p=0, 1, 2, \ldots $ denote set of all such walks with
$p$ negative steps. It is easy to see that every walk $x \in X_n (p)$
has totally $n+ 2p $  steps because $\sum x(t) = 2n. $ Therefore,
every admissible walk has  at least $n$ steps.

In view of (\ref{S_ij}), (\ref{S_21}), (\ref{mp}) and (\ref{ab}), we
have
\begin{equation}\label{sigma_p}
\beta_n(z)= \sum_{p=0}^\infty \sigma_p(n,z) \quad \text{with} \quad
\sigma_p(n,z)=\sum_{x\in X_n(p)} h(n,z),
\end{equation}
where, for $x=(x(t))_{t=1}^{k+1},$
\begin{equation}
\label{h(x,z)} h(x,z)=\frac{a^{k+1}}{(n^2-j_1^2+z)(n^2-j_2^2+z)\cdots
(n^2-j_k^2+z)}
\end{equation}
with $j_1, \ldots, j_k $ given by (\ref{j_k}).

The set $X_n (0) $ has only one element, namely the walk
\begin{equation}
\label{xi} \xi=(\xi(t))_{t=1}^n, \quad \xi(t)=2 \;\; \forall t.
\end{equation}
Therefore,
\begin{equation}
\label{xiz}  \sigma_0 (n,z)= h(\xi, z) =\frac{a^n}{(n^2-j_1^2+z)
\cdots (n^2-j_{n-1}^2+z)}
\end{equation}
with $j_k=-n+2k$, $ k=1,\cdots, n-1$. Moreover, since
\begin{align*}
\prod_{k=1}^{n-1} \left(n^2-(-n+2k)^2 \right) = 4^{n-1} [(n-1)!]^2,
\end{align*}
the following holds.
\begin{Lemma} In the above notations,
\begin{equation}\label{Sigma0}
 \sigma_0(n,0)= h(\xi,0)= \frac{4(a/4)^n}{[(n-1)!]^2}.
\end{equation}
\end{Lemma}

It is well known that
\begin{align}\label{HarmonicSum}
\sum_{k=1}^n \frac{1}{k}= \log n + \gamma + \frac{1}{2n} -
\frac{1}{12n^2} + O\left( \frac{1}{n^4} \right), \quad n\in \mathbb{N},
\end{align}
where $\gamma$ is the Euler's constant.

\begin{Lemma}
\label{lemb2} In the above notations,
\begin{equation}
\label{b2} \sigma_0 (n,z_n^\pm)=\sigma_0(n,0) \left[1- \frac{a^2 \log
n}{4n^3} - \frac{a^2 \gamma}{4n^3} + O\left(
\frac{1}{n^4}\right)\right].
\end{equation}
\end{Lemma}

\begin{proof} By (\ref{xiz}), we have
\begin{equation}\label{b3}
\sigma_0(n,z_n^\pm) =\sigma_0(n,0) \prod_{k=1}^{n-1}
\left(1+\frac{z_n^\pm}{n^2-(-n+2k)^2}\right)^{-1}.
\end{equation}
For simplicity, we set $\displaystyle b_k =
\frac{z_n^\pm}{n^2-(-n+2k)^2}= \frac{z_n^\pm}{4k(n-k)}.$ Then,
\begin{align*}
\log\left( \prod_{k=1}^{n-1} (1+b_k)^{-1} \right) &=
-\sum_{k=1}^{n-1} \log(1+b_k) = -\sum_{k=1}^{n-1} b_k + O \left(
\sum_{k=1}^{n-1} |b_k|^2 \right).
\end{align*}
Using \eqref{In0}, we obtain
\begin{align*}
\sum_{k=1}^{n-1} b_k &= \left( \sum_{k=1}^{n-1} \frac{1}{4k(n-k)} \right)
\left[ \frac{a^2}{2n^2} +  O\left( \frac{1}{n^4} \right) \right] \\
&= \frac{1}{2n} \left(\sum_{k=1}^{n-1} \frac{1}{k} \right) \left[
\frac{a^2}{2n^2} +  O\left( \frac{1}{n^4} \right) \right].
\end{align*}
By \eqref{HarmonicSum}, it follows that
\begin{equation*}
\sum_{k=1}^{n-1} b_k =\frac{a^2 \log n}{4n^3} +
\frac{a^2 \gamma}{4n^3} + O\left( \frac{1}{n^4}\right).
\end{equation*}
On the other hand, by \eqref{In0},
\begin{align*} \sum_{k=1}^{n-1} |b_k|^2 =
\left( \sum_{k=1}^{n-1} \frac{1}{[4k(n-k)]^2} \right) O\left(
\frac{1}{n^4}\right) = O \left( \frac{1}{n^4} \right).
\end{align*}
Hence,
\begin{align*} \log  \left (\prod_{k=1}^{n-1} (1+ b_k)^{-1} \right )  =
 - \frac{a^2 \log n}{4n^3} -\frac{a^2 \gamma}{4n^3}  + O\left(
 \frac{1}{n^4}\right),
\end{align*}
which implies (\ref{b2}).

\end{proof}

Next we study the ratio $\sigma_1(n,z)/\sigma_0(n,z).$
\begin{Lemma}
\label{lemR} We have
\begin{equation}
\label{r1} \sigma_1(n,z)=\sigma_0(n,z) \cdot \Phi (n,z),
\end{equation}
where
\begin{equation}
\label{r2}  \Phi (n,z)=\sum_{k=2}^{n-1} \varphi_k (n,z)
\end{equation}
with
\begin{equation}
\label{r3}  \varphi_k (n,z)=
\frac{a^2}{[n^2-(-n+2k)^2+z][n^2-(-n+2k-2)^2+z]}.
\end{equation}
\end{Lemma}

\begin{proof}
From the definition of $X_n (1) $ and
 \eqref{sigma_p} it follows that
\begin{equation}\label{sigma1}
\sigma_1(n,z)=\sum_{x\in X_n(1)} h(x,z)= \sum_{k=2}^{n-1} h(x_k,z),
\end{equation}
where $x_k$ denotes the walk with $(k+1)$'th step equal to -2, i.e.,
\begin{equation*}
x_k(t)=
\begin{cases}
~~2 & \text{if } t\neq k \\
-2 & \text{if } t=k
\end{cases},
\quad 1\leq t \leq n+2.
\end{equation*}
Now, we figure out the connection between vertices of $\xi$ and $x_k$
as follows:
\begin{equation*}
j_\alpha(x_k)= \begin{cases} j_\alpha(\xi), & 1\leq \alpha\leq k,\\
j_{k-1}(\xi) &   \alpha = k+1,\\
j_{\alpha-2}(\xi) &  k+2  \leq \alpha \leq n=2.
\end{cases}
\end{equation*}
Therefore, by \eqref{h(x,z)}
\begin{equation} \label{b12}
h(x_k,z)=
h(\xi,z)\frac{a^2}{(n^2-[j_{k-1}(\xi)]^2+z)(n^2-[j_{k}(\xi)]^2+z)}.
\end{equation}
Since $j_k(\xi)=-n+2k, \; k=2,\ldots,n-1,$ (\ref{sigma1}) and
(\ref{b12}) imply (\ref{r1}).
\end{proof}

\begin{Lemma}
\label{lemRR} In the above notations, if $z=O(1/n^2)$ then
\begin{equation}
\label{r12} \Phi (n,z) = \Phi (n,0) + O \left( 1/n^4 \right)
\end{equation}
and
\begin{equation}
\label{r13} \Phi^*(n,z) := \sum_{k=2}^{n-1} |\varphi_k (n,z)| = \Phi
(n,0) + O \left( 1/n^4 \right).
\end{equation}
Moreover,
\begin{equation}\label{r11}
\Phi (n,0) = \frac{a^2}{8n^2} + \frac{a^2\log n}{4n^3} +
\frac{a^2(\gamma-1)}{4n^3}  + O \left( 1/n^4 \right).
\end{equation}
\end{Lemma}

\begin{proof}
Since
$$
\frac{\varphi_k (n,z)}{\varphi_k (n,0)} = \left[ 1+
\frac{z}{n^2-(-n+2k)^2}\right]^{-1} \left[ 1+
\frac{z}{n^2-(-n+2k-2)^2}\right]^{-1},
$$
it is easily seen that
$$
\varphi_k (n,z)/\varphi_k (n,0) = 1+  O \left( 1/n^3 \right) \quad
\text{if} \;\; z=O(1/n^2).
$$
On the other hand, $\varphi_k (n,0)=O (1/n^2),$ so it follows that
$$
\varphi_k (n,z)-\varphi_k (n,0) =\varphi_k (n,0)\, O \left( 1/n^3
\right) = O \left( 1/n^5 \right) \;\; \text{if} \;\; z=O \left( 1/n^2
\right).
$$
Therefore, we obtain that
$$
\sum_{k=2}^{n-1} |\varphi_k (n,z)-\varphi_k (n,0)| = O \left( 1/n^4
\right) \quad \text{if} \;\; z=O(1/n^2).
$$
The latter sum dominates both $|\Phi (n,z)- \Phi (n,0)|$ and $|\Phi^*
(n,z)- \Phi (n,0)|. $ Hence,  (\ref{r12}) and (\ref{r13}) hold.\\

Next we prove (\ref{r11}). Since
\begin{equation*}
\Phi (n,0)=\sum_{k=2}^{n-1} \frac{a^2}{16(k-1)k(n-k)(n+1-k)},
\end{equation*}
by using the identities
\begin{equation*}
\frac{1}{k(n-k)}= \frac {1}{n} \left( \frac{1}{k}+\frac{1}{n-k} \right),
\quad \frac{1}{(k-1)(n+1-k)}= \frac {1}{n} \left( \frac{1}{k-1}+\frac{1}{n+1-k} \right)
\end{equation*}
we obtain
\begin{equation}
\label{Di} \Phi (n,0) = \frac{a^2}{16n^2} \sum_{i=1}^4 D_i (n),
\end{equation}
where
\begin{equation*}
D_1(n)= \sum_{k=2}^{n-1} \frac{1}{k(k-1)},
\quad D_2(n)= \sum_{k=2}^{n-1} \frac{1}{(n-k)(n+1-k)},
\end{equation*}
\begin{equation*}
D_3(n)= \sum_{k=2}^{n-1} \frac{1}{k(n+1-k)},
\quad D_4(n)= \sum_{k=2}^{n-1} \frac{1}{(k-1)(n-k)}.
\end{equation*}
The change of summation index  $m=n+1-k$ shows that $D_2(n)=D_1(n),$
and we have
\begin{equation}\label{D_1}
D_1(n)=\sum_{k=2}^{n-1} \left (\frac{1}{k-1}- \frac{1}{k} \right )=1-
\frac{1}{n-1}= 1 - \frac{1}{n} + O \left( \frac{1}{n^2}\right).
\end{equation}
Moreover, since
\begin{equation*}
D_3(n)= \frac{1}{n+1} \left( \sum_{k=2}^{n-1} \frac{1}{k} +
\sum_{k=2}^{n-1} \frac{1}{n+1-k} \right) =\frac{2}{n+1}
\sum_{k=2}^{n-1} \frac{1}{k},
\end{equation*}
by \eqref{HarmonicSum} we obtain that
\begin{equation}\label{D_3}
D_3(n)=\frac{2 \log n}{n} +\frac{2(\gamma -1)}{n}-
\frac{2 \log n}{n^2} + O \left( \frac{1}{n^2}\right).
\end{equation}
Similarly,
\begin{equation*}
D_4(n)= \frac{1}{n-1} \left( \sum_{m=1}^{n-2} \frac{1}{m} +
\sum_{m=1}^{n-2} \frac{1}{n-m-1} \right)= \frac{2}{n-1}
\sum_{m=1}^{n-2} \frac{1}{m},
\end{equation*}
and \eqref{HarmonicSum} leads to
\begin{equation}\label{D_4}
D_4(n)=\frac{2 \log n}{n} + \frac{2\gamma}{n} +
\frac{2\log n}{n^2} +  O \left( \frac{1}{n^2}\right).
\end{equation}
Hence, in view of \eqref{Di}--\eqref{D_4}, we obtain \eqref{r11}.
\end{proof}

\begin{Proposition}
We have
\begin{equation}\label{beta}
\beta_n(z_n^\pm)=\sigma_0(n,0) \left[  1 + \frac{a^2}{8n^2} -
\frac{a^2}{4n^3} + O \left( \frac{1}{n^4}\right) \right] .
\end{equation}
\end{Proposition}

\begin{proof}
From (\ref{b2}), (\ref{r1}), (\ref{r12}) and (\ref{r11}) it follows
immediately that
$$
\sigma_1(n,z_n^\pm)+\sigma_0(n,z_n^\pm) = \sigma_0(n,0) \left[  1 +
\frac{a^2}{8n^2} - \frac{a^2}{4n^3} + O \left( \frac{1}{n^4}\right)
\right] .
$$
Since $\beta_n (z) = \sum_{p=0}^\infty \sigma_p (n,z), $ in view of
(\ref{b2}) to complete the proof it is enough to show that
\begin{equation}
\label{S2} \sum_{p=2}^\infty \sigma_p(n,z_n^\pm)= \sigma_0(n,z_n^\pm)
\, O\left(\frac{1}{n^4}\right).
\end{equation}

Next we prove (\ref{S2}). Recall that $\sigma_p (n,z) = \sum_{x\in
X_n(p)} h(x,z).$ Now we set
$$
\sigma^*_p (n,z) = \sum_{x\in X_n(p)} |h(x,z)|.
$$
We are going to show that there is an absolute constant $C>0$ such
that
\begin{equation}
\label{S3} \sigma^*_p (n,z_n^\pm) \leq \sigma^*_{p-1} (n,z_n^\pm)
\cdot \frac{C}{n^2}, \quad p\in \mathbb{N}, \;\; n \geq N_0.
\end{equation}
Since $\sigma_0 (n,z)$ has one term only, we have $\sigma^*_0
(n,z)=|\sigma_0 (n,z)|.$

Let $p \in \mathbb{N}.$ To every walk $x \in X_n(p)$ we assign a pair
$(\tilde{x},j)$, where $\tilde{x} \in X_n(p-1)$ is the walk that we
obtain after dropping the first cycle $ \{ +2, -2 \}$ from $x,$ and
$j$ is the vertex of $x$ where the first negative step of $x$ is
performed. In other words, we consider the map
\begin{equation*}
\varphi: X_n(p)\longrightarrow X_n(p-1) \times I, \qquad I=\{-n+4, -n+6, \ldots, n-2\},
\end{equation*}
defined by $\varphi(x)=(\tilde{x},j)$, where
\begin{equation*}
\tilde{x}(t)=
\begin{cases}
x(t) & \text{if } 1\leq t\leq k-1 \\
x(t+2) & \text{if } k \leq t \leq n+2p-2
\end{cases},
\end{equation*}
where $k=\min\{t:x(t)=2,~x(t+1)=-2\}$ and $j=-n+2k$.

The map $\varphi$ is clearly injective, and moreover, we have
\begin{equation} \label{S5}
h(x,z)= h(\tilde{x},z)\frac{a^2}{(n^2-j^2+z)(n^2-(j-2)^2+z)}.
\end{equation}
Since the mapping $\varphi$ is injective, from (\ref{r3}),
(\ref{r13}) and \eqref{S5} it follows that
\begin{equation}\label{S6}
\sigma^*_p(n,z) \leq \sigma^*_{p-1}(n,z)\cdot \Phi^*(n,z).
\end{equation}
Hence, by \eqref{r13} and (\ref{r11}), we obtain that (\ref{S3})
holds.

From (\ref{S3}) it follows (since $\sigma^*_0 (n,z_n^\pm)=|\sigma_0
(n,z_n^\pm)|$) that
$$
\sigma^*_p(n,z_n^\pm) \leq |\sigma_0 (n,z_n^\pm)| \cdot \left
(\frac{C}{n^2}\right)^p.
$$
Hence, (\ref{S2}) holds, which completes the proof.
\end{proof}

\begin{Theorem}
The  Mathieu operator
\begin{equation*}
L(y)=-y''+2a \cos{(2x)}y, \quad a\in \mathbb{C}, \;\; a\neq 0,
\end{equation*}
considered with periodic or anti-periodic boundary conditions has,
close to $n^2$ for large enough $n$, two periodic (if $n$ is even) or
anti-periodic (if $n$ is odd) eigenvalues $\lambda_n^-$,
$\lambda_n^+$. For fixed nonzero $a\in \mathbb{C}$,
\begin{equation}\label{G}
\lambda_n^+ -\lambda_n^-  = \pm \frac{8(a/4)^n}{[(n-1)!]^2} \left[1 -
\frac{a^2}{4n^3}+ O \left( \frac{1}{n^4}\right)\right], \quad n \to
\infty.
\end{equation}
\end{Theorem}

\begin{proof} The basic equation \eqref{be}
splits into two equations
\begin{align}\label{E1}
z- \alpha_n(z)- \beta_n(z)&=0, \\ \label{E2} z- \alpha_n(z)+
\beta_n(z)&=0.
\end{align}
In view of  (\ref{BE1}) and  \eqref{In7}, it follows that  for large
enough $n$
\begin{equation*}
|z- \alpha_n(z) \pm \beta_n(z)| < |z| \quad \text{if} \;\;  |z|=1.
\end{equation*}
Hence, for large enough $n,$ each of the equations \eqref{E1} and
\eqref{E2} has only one root in the unit disc due to Rouche's
theorem.

On the other hand, by Lemmas~\ref{loc} and \ref{lem1}, for large
enough $n$ the basic equation has exactly two roots $z_n^-, z_n^+$ in
the unit disc, so either $z_n^-$ is the root of (\ref{E1}) and $
z_n^+$ is the root of (\ref{E2}), or $z_n^+$ is the root of
(\ref{E1}) and $ z_n^-$ is the root of (\ref{E2}). Therefore,  we
obtain
\begin{equation*}
z_n^+ - z_n^- -[\alpha_n(z_n^+)- \alpha_n(z_n^-)]  = \pm [
\beta_n(z_n^+) + \beta_n(z_n^-)].
\end{equation*}
Now, \eqref{D1} and \eqref{beta} imply, with $\gamma_n= \lambda_n^+ -
\lambda_n^-,$
\begin{equation*}
\gamma_n \left[1 +\frac{a^2}{8n^2} + O \left( \frac{1}{n^4}\right)
\right] = \pm 2\sigma_0(n,0) \left[1+ \frac{a^2}{8n^2} -
\frac{a^2}{4n^3} + O \left( \frac{1}{n^4} \right)  \right].
\end{equation*}
Therefore,
\begin{align*}
\gamma_n &= \pm 2\sigma_0(n,0) \left[1+ \frac{a^2}{8n^2} -
\frac{a^2}{4n^3} + O \left( \frac{1}{n^4} \right)  \right]
\left[1 -\frac{a^2}{8n^2} + O \left( \frac{1}{n^4}\right) \right] \\
&= \pm 2\sigma_0(n,0) \left[1 - \frac{a^2}{4n^3} + O \left(
\frac{1}{n^4}\right)\right].
\end{align*}
Hence, in view of \eqref{Sigma0}, \eqref{G} holds.
\end{proof}

\end{document}